\documentclass[11pt]{amsart}

\usepackage{amsmath, amssymb, amsthm,mathtools,tikz,tikz-cd,graphicx,color,stmaryrd,hyperref,enumerate,mathdots,mathrsfs,dynkin-diagrams}
\usepackage[new]{old-arrows}
\usepackage[shortlabels]{enumitem}

\usepackage[utf8]{inputenc}

\usepackage{appendix}
\usepackage[margin=1in]{geometry}

\usepackage[english]{babel}

\hypersetup{
    colorlinks,
    citecolor=blue,
    filecolor=blue,
    linkcolor=blue,
    urlcolor=blue
}

\DeclareSymbolFont{extraup}{U}{zavm}{m}{n}
\DeclareMathSymbol{\varheart}{\mathalpha}{extraup}{86}

\newcommand{\shear}{{\mathbin{\mkern-6mu\fatslash}}}

\newcommand{\fg}{\mathfrak{g}}

\newcommand{\fn}{\mathfrak{n}}
\newcommand{\fb}{\mathfrak{b}}
\newcommand{\fz}{\mathfrak{z}}
\newcommand{\bG}{\mathbb{G}}
\newcommand{\bL}{\mathbb{L}}
\newcommand{\B}{\mathbb{B}}
\newcommand{\bA}{\mathbb{A}}
\newcommand{\cF}{\mathcal{F}}
\newcommand{\cG}{\mathcal{G}}

\newcommand{\Z}{{\mathbb Z}}
\newcommand{\cL}{{\mathcal L}}
\newcommand{\cM}{{\mathcal M}}
\newcommand{\cC}{{\mathrm C}}
\newcommand{\cY}{{\mathcal Y}}
\newcommand{\cO}{{\mathcal O}}

\newcommand{\pr}{\operatorname{pr}}

\newcommand{\reg}{\mathrm{reg}}
\newcommand{\irreg}{\mathrm{irreg}}
\newcommand{\et}{\mathrm{\acute{e}t}}

\newcommand{\cN}{{\mathcal N}}
\newcommand{\ft}{{\mathfrak{t}}}

\newcommand{\ab}{{\mathrm{ab}}}

\newcommand{\cE}{\mathcal{E}}
\newcommand{\triv}{\mathrm{triv}}
\newcommand{\spec}{\mathrm{spec}}
\newcommand{\pt}{\mathrm{pt}}
\newcommand{\dR}{\mathrm{dR}}

\newcommand{\Coh}{\operatorname{Coh}}
\newcommand{\Rep}{\operatorname{Rep}}
\newcommand{\Spec}{\operatorname{Spec}}
\newcommand{\Vect}{\operatorname{Vect}}
\newcommand{\Ge}{\mathrm{Ge}}

\newcommand{\Sym}{\operatorname{Sym}}
\newcommand{\Sing}{\operatorname{Sing}}

\newcommand{\SL}{{\mathrm{SL}}}

\newcommand{\cP}{{\mathcal{P}}}

\newcommand{\Bun}{{\mathrm{Bun}}}
\newcommand{\Hck}{{\mathrm{Hck}}}
\newcommand{\Kos}{{\mathrm{Kos}}}
\newcommand{\cHck}{\mathcal{H}\mathrm{ck}}
\newcommand{\Sat}{\operatorname{Sat}}
\newcommand{\Sph}{\operatorname{Sph}}
\newcommand{\Whit}{\operatorname{Whit}}

\newcommand{\Eis}{\operatorname{Eis}}
\newcommand{\CT}{\operatorname{CT}}
\newcommand{\Poinc}{\operatorname{Poinc}}
\newcommand{\pp}{{\mathsf{p}}}
\newcommand{\qq}{{\mathsf{q}}}
\newcommand{\rr}{{\mathsf{r}}}
\newcommand{\abc}{{\mathrm{abc}}}
\newcommand{\Nilp}{\mathrm{Nilp}}

\newcommand{\qt}{\operatorname{quasi-temp}}
\newcommand{\mat}{\operatorname{max-anti-temp}}
\newcommand{\sing}{\operatorname{sing}}
\newcommand{\oblv}{\operatorname{oblv}}

\DeclareMathOperator{\Lie}{Lie}
\DeclareMathOperator{\ad}{ad}

\DeclareMathOperator{\Hom}{Hom}
\DeclareMathOperator{\QCoh}{QCoh}

\DeclareMathOperator{\Dmod}{D-mod}
\DeclareMathOperator{\IndCoh}{IndCoh}
\DeclareMathOperator{\LS}{LS}
\DeclareMathOperator{\Map}{Map}

\newtheorem{thm}{Theorem}[section]
\newtheorem{lemma}[thm]{Lemma}
\newtheorem{cor}[thm]{Corollary}
\newtheorem{prop}[thm]{Proposition}

\newtheorem{example}[thm]{Example}

\theoremstyle{remark}
\newtheorem*{remark}{Remark}

\theoremstyle{definition}
\newtheorem{defn}[thm]{Definition}

\title[The constant sheaf under geometric Langlands]{The image of the constant sheaf under\\ the geometric Langlands equivalence}

\makeatletter
\let\@wraptoccontribs\wraptoccontribs
\makeatother

\author{Kenta Suzuki}
\address{Department of Mathematics, Princeton University, Fine Hall, Washington Road, Princeton, NJ 08544-1000, USA}
\email{kjsuzuki@princeton.edu}

\begin{document}

\begin{abstract}
Let $X$ be a smooth projective curve and let $G$ be a reductive group over an algebraically closed field of characteristic zero. We compute the image of arbitrary finite-rank local systems on $\Bun_G(X)$ under the geometric Langlands equivalence, confirming a conjecture of V. Lafforgue in the case of the constant sheaf.
\end{abstract}

\maketitle

\section{Introduction} 
Let $X$ be a smooth projective curve over an algebraically closed field $e$ of characteristic zero. Let $G$ be a reductive group over $e$ and let $\check G$ be its dual group. We furthermore choose a spin structure $\omega_X^{1/2}$ on $X$.
The geometric Langlands equivalence \cite{GLC1,GLC2,GLC3,GLC4,GLC5} states\footnote{The reader may be accustomed to seeing $\Dmod_{\frac12}(\Bun_G)$, the category of D-modules twisted by $\det_{\Bun_G}^{1/2}$, on the left-hand side. The choice of spin structure on $X$ gives a choice of square root of $\det_{\Bun_G}$ which identifies $\Dmod_{\frac12}(\Bun_G)$ with $\Dmod(\Bun_G)$, as discussed in \cite[\S1.1.5]{GLC1}.
}
    \[
\mathbb L_G\colon\Dmod(\Bun_G)\simeq\IndCoh_{\Nilp}(\LS_{\check G}).
\]
We confirm an unpublished conjecture of V. Lafforgue and compute the image of $\omega_{\Bun_G}$ under $\mathbb L_G$. 
\begin{remark}
    The coarse geometric Langlands functor $\mathbb L_{G,\mathrm{coarse}}\colon\Dmod(\Bun_G)\to\QCoh(\LS_{\check G})$ kills the object $\omega_{\Bun_G}$, as noted in \cite[\S2.1.4]{GLC1}. Thus, the computation of $\mathbb L_G(\omega_{\Bun_G})$ has to do with \emph{cohomological subtleties} of the difference between quasi-coherent sheaves and ind-coherent sheaves.
\end{remark}
\subsection{Lafforgue's expectation}\label{subsec:intro-laff}
Fix a maximal torus and Borel $\check T\subset \check B\subset\check G$, and let $\check N$ be the unipotent radical of $\check B$. For simplicity, in the introduction we \emph{assume $\check\rho$ exists as a cocharacter $\bG_m\to\check T$} (which holds, for example, when $G$ is simply connected). Lafforgue considers the correspondence
\begin{equation}\label{eq:intro-corr}\begin{tikzcd}
    &\LS_{\check N}/\check\rho(\bG_m)\arrow[swap]{dl}{\iota}\arrow{dr}{\Sigma}\\
    \LS_{\check G}&&\Omega\bA^1/\bG_m.
\end{tikzcd}\end{equation}
Here, $\iota$ is the composition (this is where the existence of $\check\rho$ is used)
\[
\LS_{\check N}/\check\rho(\bG_m)\to\LS_{\check N\rtimes_{\check\rho}\bG_m}\to\LS_{\check G}
\]and $\Sigma$ is the composition
\[
\LS_{\check N}\to\LS_{\bG_a}\to\Omega\bA^1,
\]
where $\LS_{\check N}\to\LS_{\bG_a}$ is induced from a non-degenerate character $\check N\to\bG_a$ and $\LS_{\bG_a}\to\Omega\bA^1$ arises from the isomorphism $\LS_{\bG_a}\simeq\cC^\bullet_\dR(X)[1]$ and the trace map $\cC^\bullet_\dR(X)\to e[-2]$ (see Lemma~\ref{lem:LS_Ga}).

There exists a certain multiplicative object $\exp^\spec\in\IndCoh(\Omega\bA^1/\bG_m)$ described in \cite[\S11.6.3]{relative-langlands} (see Definition~\ref{def:spectral-exponential}). Define the \emph{spectral Poincar\'e sheaf} to be
\begin{equation}\label{intro-eq:defn-of-poinc-spec}
\Poinc^\spec:=\iota_*^{\IndCoh}\Sigma^!(\exp^\spec).
\end{equation}
\begin{thm}[Theorem~\ref{prop:image-of-constant-sheaf}]\label{intro-thm1}
    Suppose $\check\rho$ is a cocharacter. Then there is an isomorphism \[\mathbb L_G(\omega_{\Bun_G})\simeq\Poinc^\spec.\]
\end{thm}
In fact, we also compute the image of an arbitrary finite-rank local system on $\Bun_G$ in Theorem~\ref{thm:laf1}.
\begin{remark}
    There is a functor in \cite{GLC2} denoted $\Poinc^\spec$ which is also a spectral analog of the Poincar\'e series construction. This has nothing to do with our spectral Poincar\'e sheaf!
\end{remark}

Equivalently, upon dualizing, we obtain a spectral description of the compactly supported de Rham cohomology functor $\cC_{\dR,c}^\bullet(\Bun_G,-)$. We let the \emph{spectral Whittaker functor} $\Whit^\spec$ be the composition
\[
\IndCoh_\Nilp(\LS_{\check G})\xrightarrow{\iota^!}\IndCoh(\LS_{\check N}/\check\rho(\bG_m))\xrightarrow{\Sigma_*^{\IndCoh}}\IndCoh(\Omega\bA^1/\bG_m)\xrightarrow{\sing}\Vect,
\]
where $\sing$ is the quotient by the subcategory $\QCoh(\Omega\bA^1/\bG_m)\subset \IndCoh(\Omega\bA^1/\bG_m)$.
\begin{thm}\label{intro-thm2}
    Suppose $\check\rho$ is a cocharacter. Then (up to a cohomological shift) the following diagram commutes:
    \begin{equation}\label{eq:whittaker-cohomology}
\begin{tikzcd}
    \Dmod(\Bun_G)\arrow{rr}{\mathbb L_G}\arrow[swap]{dr}{\cC^\bullet_{\dR,c}(\Bun_G,-)}&&\IndCoh_\Nilp(\LS_{\check G})\arrow{dl}{\Whit^\spec}\\
    &\Vect.
\end{tikzcd}
\end{equation}
\end{thm}
\begin{remark}
    Theorem~\ref{intro-thm1} also holds in other sheaf theories for formal reasons. 
Indeed, the object $\Poinc^\spec$ lies in the category $\IndCoh_\Nilp(\LS_{\check G}^{\mathrm{restr}})$ where $\LS_{\check G}^{\mathrm{restr}}$ is the stack of $\check G$-local systems with restricted variation, in the sense of \cite{AGKRRV} (see the proof of Corollary~\ref{cor:poincare-factors2}). Gaitsgory and Raskin \cite{GLC1} deduce the Betti geometric Langlands equivalence from the de Rham statement via the stack $\LS_{\check G}^{\mathrm{restr}}$, so the formula \eqref{intro-eq:defn-of-poinc-spec} still computes the image of the (Betti) constant sheaf on $\Bun_G$. In the $\ell$-adic setting, Gaitsgory and Raskin \cite{GLC-l-adic} produce a fully faithful embedding of the category of $\ell$-adic sheaves with nilpotent singular support on $\Bun_G$ in positive characteristic into $\operatorname{IndCoh}_{\mathrm{Nilp}}(\LS_{\check G}^{\mathrm{restr}})$. Their construction is via nearby cycles, which simply takes the constant sheaf to the constant sheaf, so \eqref{intro-eq:defn-of-poinc-spec} is still the image of the ($\ell$-adic) constant sheaf on $\Bun_G$.
\end{remark}

\begin{remark}
    When $\check\rho$ is \emph{not} a cocharacter, the natural morphism $\LS_{\check N}\to\LS_{\check G}$ may not factor through $\LS_{\check N}/\check\rho(\bG_m)$. For example, when $\check G=\SL_2$ the stack $\LS_{\check N}/\check\rho(\bG_m)$ classifies, for a test scheme $S$, a line bundle $\cL$ on $S$ together with an extension
    \begin{equation}\label{eq:example-SL2}
    \cL\boxtimes\omega_X\to \cE\to \cO_S\boxtimes\omega_X
    \end{equation}
    in $\QCoh(S)\otimes\Dmod(X)$. There is no morphism to $\LS_{\SL_2}$, since $\det(\cE)\simeq \cL\boxtimes\omega_X$ cannot be trivialized. Later on, we will instead consider the stack $\LS_{\check Z\check N}/\check\rho(\bG_m)$, which classifies a line bundle $\cL$ on $S$, an extension as in \eqref{eq:example-SL2}, \emph{and} a line bundle $\cM\in\QCoh(S)\otimes\Dmod(X)$ with an isomorphism $\cM^{\otimes 2}\simeq \cL\boxtimes\omega_X$. Then $\cM^{-1}\otimes\cE$ has trivial determinant, giving a morphism $\LS_{\check Z\check N}/\check\rho(\bG_m)\to\LS_{\check G}$.  
\end{remark}

\subsection{Relation to relative Langlands duality} In the framework of \cite{relative-langlands}, the dual of $\check G$ acting on a point is $G/(N,\psi)$ where $\psi$ is a non-degenerate character of $N$. For example, by definition, the Langlands functor $\mathbb L_G$ sits in a commutative diagram
\[
\begin{tikzcd}
    \Dmod(\Bun_G)\arrow{rr}{\mathbb L_G}\arrow[swap]{dr}{\mathrm{coeff}}&&\IndCoh_\Nilp(\LS_{\check G})\arrow{dl}{\Gamma(\LS_{\check G},-)}\\
    &\Vect,
\end{tikzcd}
\]
so the global sections functor is intertwined with Whittaker coefficients. Under the philosophy of relative Langlands, the roles of $G$ and $\check G$ may be swapped, so the ``global sections" functor (i.e., compactly supported de Rham cohomology) on the automorphic side is expected to intertwine with a \emph{spectral Whittaker} functor. This is precisely what \eqref{eq:whittaker-cohomology} states. The analogy between the exponential D-module, which is used to define $\mathrm{coeff}$, and the spectral exponential sheaf, which is used to define $\Whit^\spec$, is further discussed in the remark following Definition~\ref{def:spectral-exponential}.

\subsection{Outline of the proof} Let us assume $G$ is semisimple for simplicity. We let $\cF:=\mathbb L_G^{-1}(\Poinc^\spec)$, which is the mystery object we hope to show is equal to the constant sheaf $\omega_{\Bun_G}$ (up to cohomological shifts). The first step will be to check $\cF$ is quasi-lisse, i.e., its singular support lies in the zero-section. By \cite[Theorem~A]{quasi-temp}, this amounts to checking that the dual functor
\[
\Whit^\spec\colon\IndCoh_{\Nilp}(\LS_{\check G})\to\Vect
\]
is zero on the subcategory $\IndCoh_{\Nilp_\irreg}(\LS_{\check G})$. Let $\pp^-\colon\Bun_{B^-}\to\Bun_G$ and $\rr\colon\Bun_T\to\Bun_{B^-}$ be the natural morphisms. Then, in Proposition~\ref{prop:eisenstein-and-delta} we show $\CT_*^-\simeq(\pp^-\circ\rr)^!$ up to cohomological shifts on the subcategory $\Dmod_0(\Bun_G)\subset\Dmod(\Bun_G)$. Moreover, since $\pp^-\circ\rr$ is surjective on both connected components and fundamental groups, to check $\cF\simeq\omega_{\Bun_G}$ it suffices to check $(\pp^-\circ\rr)^!\cF\simeq\CT^-_*\cF\simeq\omega_{\Bun_T}$. By the compatibility of constant term functors on the automorphic and spectral sides from \cite{GLC3}, we reduce to computing $\CT^{-,\spec}(\Poinc^\spec)$. This is computed in Proposition~\ref{prop:spectral-image}.

In \S\ref{sec:preliminaries} we cover preliminaries. In \S\ref{sec:compatibility-of-geometric-langlands-and-isogenies} we prove a compatibility of the geometric Langlands equivalence with isogenies as in \cite[\S5.4]{GLC5}, which will be used later to deduce the image of arbitrary local systems under $\mathbb L_G$ from the computation for the constant sheaf $\omega_{\Bun_G}$. To formulate the compatibility in full generality we prove a variant of geometric Langlands for disconnected tori, which may be of independent interest. In \S\ref{sec:lafforgue-expectation} we use the compatibility proved in \S\ref{sec:compatibility-of-geometric-langlands-and-isogenies} to fully formulate our main theorem, enhancing Lafforgue's original conjecture. In \S\ref{sec:some-computations} we compute the images of all relevant objects under constant term functors, which is the basic computation that allows us to match the automorphic and spectral sides, as outlined above. Finally, in \S\ref{sec:conclusion} we combine our ingredients to prove our main theorem.

\subsection{Acknowledgments} I thank Sam Raskin for introducing me to the problem and patiently explaining the basics of geometric Langlands. He also carefully read an earlier version of this paper and had many helpful comments and corrections. I thank Wyatt Reeves for explaining Koszul duality and singular supports of coherent sheaves. 

\section{Preliminaries}\label{sec:preliminaries}

\subsection{Notation} 
\begin{itemize}
    \item Let $g$ denote the genus of the curve $X$.
    \item For a reductive or unipotent group $H$ over $e$, we let $\delta_H=\dim(\Bun_H)=(g-1)\dim(H)$. 
    \item Let $\check Z$ denote the center of $\check G$.
\item For a finite-dimensional vector space $V$, we let $\Omega V=0\times_V0$, where the fiber product is derived.
\item For a category $\mathcal C$ with an action of $\Rep(\bG_m)$, let $\mathcal C^\shear$ denote the \emph{shearing} or \emph{shift of grading}, described in \cite[\S A.2]{singular-support}, \cite[\S2.1]{spectral-gluing}, and \cite[Definition~6.3.2]{relative-langlands}. 
\end{itemize}

\subsection{Miraculous duality and Serre duality}\label{subsec:duality} 
Miraculous duality \cite{miraculous-duality,compact-generation} is an equivalence\[
\mathbb D_{\Bun_G,!}\colon\Dmod(\Bun_G)^\vee\simeq\Dmod(\Bun_G)
\]
given by the kernel $\Delta_!e_{\Bun_G}\in\Dmod(\Bun_G\times\Bun_G)$.
In other words, given an object $F$ of $\Dmod(\Bun_G)^\vee$, i.e., a colimit-preserving functor $F\colon\Dmod(\Bun_G)\to\Vect$, the miraculous duality functor $\mathbb D_{\Bun_G,!}$ attaches the image of $\Delta_!e_{\Bun_G}$ under the composition
\[
\Dmod(\Bun_G\times\Bun_G)\simeq\Dmod(\Bun_G)\otimes\Dmod(\Bun_G)\xrightarrow{F\otimes\operatorname{Id}}\Dmod(\Bun_G).
\]
On the other hand, the spectral category $\IndCoh_\Nilp(\LS_{\check G})$ is also self-dual by Serre duality:
\begin{align*}
\mathbb D_{\mathrm{Serre}}^{-1}\colon\IndCoh_\Nilp(\LS_{\check G})&\to \IndCoh_\Nilp(\LS_{\check G})^\vee\\
\cF&\mapsto\Gamma(\LS_{\check G},\cF\otimes^!-).
\end{align*}
The two duality functors are compatible up to a Chevalley involution $\tau$ and a cohomological shift; see \cite[\S1.1.8]{kevin}.
\begin{prop}
\label{lem:duality-intertwined}
    There is a commutative diagram
    \[
    \begin{tikzcd}
        \Dmod(\Bun_G)^\vee\arrow{r}{\mathbb D_{\Bun_G,!}}&\Dmod(\Bun_G)\arrow{dr}{\mathbb L_G}\\
        \IndCoh_\Nilp(\LS_{\check G})^\vee\arrow{r}{\mathbb D_{\mathrm{Serre}}}\arrow{u}{\mathbb L_G^\vee}&\IndCoh_\Nilp(\LS_{\check G})\arrow{r}{\tau[d]}&\IndCoh_\Nilp(\LS_{\check G}),
    \end{tikzcd}
    \]
    where $\tau$ denotes the Chevalley involution and $d=4\delta_G+2\delta_N$. 
\end{prop}
We will also use the following general fact from \cite[Proposition~4.4.11]{finiteness}.
\begin{lemma}\label{lem:serre-duality}
    Let $\mathcal X$ and $\mathcal Y$ be QCA algebraic stacks. For any morphism $f\colon\mathcal X\to\mathcal Y$, there is a commutative diagram
    \[
    \begin{tikzcd}
        \IndCoh(\mathcal X)^\vee\arrow{r}{\mathbb D_{\mathrm{Serre}}}\arrow{d}{(f^!)^\vee}&\IndCoh(\mathcal X)\arrow{d}{f_*^{\IndCoh}}\\
        \IndCoh(\mathcal Y)^\vee\arrow{r}{\mathbb D_{\mathrm{Serre}}}&\IndCoh(\mathcal Y).
    \end{tikzcd}
    \]
\end{lemma}

\subsection{Koszul duality}
Koszul duality, as stated in \cite[Corollary~5.1.10]{singular-support}, is the following.
\begin{lemma}
    Let $V$ be a finite-dimensional vector space over $e$. There is an equivalence
    \begin{align*}
\Kos_V\colon \IndCoh(\Omega V)&\simeq \Sym(V[-2])\operatorname{-mod}\\M&\mapsto \Hom_{\Omega V}(e,M).
\end{align*}
\end{lemma}
In the $\bG_m$-equivariant setting, we have the following, which is essentially the definition of coherent singular support.
\begin{lemma}[{\cite[Lemma~4.3.1]{tempered-d-modules}}]
    Let $V$ be a finite-dimensional vector space over $e$. For a $\bG_m$-equivariant closed subscheme $\Lambda\subset V^\vee$ there is an equivalence
\[
\IndCoh_\Lambda(\Omega V/\bG_m)\simeq\IndCoh((V^\vee)^\land_\Lambda/\bG_m)^\shear.
\]
\end{lemma}
\begin{remark}
    In particular, when $\Lambda=V^\vee$ we see
$\IndCoh(\Omega V/\bG_m)\simeq\IndCoh(V^\vee/\bG_m)^\shear$.
\end{remark}
Let us also recall the following general fact about Koszul duality:
\begin{lemma}\label{lem:koszul-commutative}
    Given a linear map $f\colon V\to W$ with dual $f^\vee\colon W^\vee\to V^\vee$, there is a commutative diagram
\[\begin{tikzcd}
    \IndCoh(\Omega W/\bG_m)\arrow{d}{\Kos_W}\arrow{r}{\Omega f^!}&\IndCoh(\Omega V/\bG_m)\arrow{d}{\Kos_V}\\
    \IndCoh(W^\vee/\bG_m)^\shear\arrow{r}{(f^\vee)_*^\shear}&\IndCoh(V^\vee/\bG_m)^\shear.
\end{tikzcd}\]
\end{lemma}
\begin{proof}Indeed, for $M\in\IndCoh(\Omega W)$,
\[
    \Kos_V\circ\Omega f^!(M)\simeq \Hom_{\Omega V}(e,\Omega f^! M)\simeq \Hom_{\Omega W}(\Omega f_*^{\IndCoh}e,M)\simeq \Hom_{\Omega W}(e,M).
\]
This is a $\Sym(W[-2])$-module, but we only consider it as a $\Sym(V[-2])$-module, which amounts to the pushforward $(f^\vee)_*^{\IndCoh}$.
\end{proof}

Note that there is an exact sequence of categories
\[
\IndCoh((\bA^1)^\land_0/\bG_m)\xrightarrow{i_*}\IndCoh(\bA^1/\bG_m)\xrightarrow{j^!}\Vect
\]
corresponding to the closed and open substacks $i\colon 0/\bG_m\hookrightarrow\bA^1/\bG_m\hookleftarrow\mathrm{pt}: \!j$. Upon shearing and Koszul duality, this gives an exact sequence
\[
\QCoh(\Omega\bA^1/\bG_m)\to\IndCoh(\Omega\bA^1/\bG_m)\xrightarrow{(j^!)^\shear}\Vect.
\]
We are particularly interested in the following object of $\IndCoh(\Omega\bA^1/\bG_m)$; see also \cite[\S{}A.2]{relative-langlands}.
\begin{defn}\label{def:spectral-exponential}
    Let $j\colon\mathrm{pt}\hookrightarrow\bA^1/\bG_m$ be the open immersion. Then the \emph{spectral exponential sheaf} $\exp^\spec\in\IndCoh(\Omega\bA^1/\bG_m)$ is the image of $j_*^\shear(e)\in\IndCoh(\bA^1/\bG_m)^\shear$ under Koszul duality.
\end{defn}
The spectral exponential sheaf can alternatively be described as follows. The object $e\in\Coh(\Omega\bA^1/\bG_m)$ lies in an exact triangle $e(1)[1]\to\cO_{\Omega\bA^1/\bG_m}\to e$, which upon rotation gives a morphism $\xi\colon e\to e(1)[2]$. Then,
\[
\exp^\spec=\operatorname{colim}\big(e\xrightarrow\xi e(1)[2]\xrightarrow\xi e(2)[4]\xrightarrow\xi\cdots\big).
\]
\begin{remark}
    $\exp^\spec$ is the unique nontrivial multiplicative ind-coherent sheaf on $\Omega\bA^1/\bG_m$, in analogy with the exponential D-module on $\bA^1$. In other words, the only objects $A\in\IndCoh(\Omega\bA^1/\bG_m)$ with compatible isomorphisms
    \[+^!A\simeq A\boxtimes A\in\IndCoh(\Omega\bA^2/\bG_m),\ 
    0^!A\simeq\cO_{B\bG_m}\in\IndCoh(B\bG_m)\]
are $\omega_{\Omega\bA^1/\bG_m}$ and $\exp^\spec$. Indeed, under Koszul duality such an object corresponds to an idempotent graded $e[\xi]$-algebra, where $|\xi|=2$. They are classified by \cite{localizing}.
\end{remark}
\begin{defn}
    Denote the functor $(j^!)^\shear\colon\IndCoh(\Omega\bA^1/\bG_m)\to\Vect$ by $\sing$, which we call the \emph{singularity functor}.
\end{defn}
In the language of \cite[\S H.2]{singular-support}, the functor $(j^!)^\shear$ presents $\Vect$ as the category of singularities of $\Omega\bA^1/\bG_m$.

More generally, for a prestack $\mathcal Y$ with a $\bG_m$-action, we let $\sing_{\mathcal Y}$ denote the composition
\begin{equation}\label{eq:singularity-functor}
\IndCoh\!\big((\mathcal Y\times\Omega\bA^1)/\bG_m\big)\simeq\IndCoh\!\big((\mathcal Y\times\bA^1)/\bG_m\big)^\shear\xrightarrow{(j^!)^\shear}\IndCoh(\mathcal Y).
\end{equation}
Dually, we let $\exp^\spec_{\mathcal Y}$ denote the composition
\begin{equation}\label{eq:singularity-functor2}
\IndCoh(\mathcal Y)\xrightarrow{(j_*)^\shear}\IndCoh\!\big((\mathcal Y\times\bA^1)/\bG_m\big)^\shear\simeq
\IndCoh\!\big((\mathcal Y\times\Omega\bA^1)/\bG_m\big).
\end{equation}

\subsection{Singular support and non-characteristic morphisms}
We review the notion of the singular support of D-modules on a smooth algebraic stack $\cY$ as defined in \cite[Appendix~\S{}F.6.3]{AGKRRV} (see also \cite[\S2.2.3]{quasi-temp}). We take as given the category $\Dmod_\cN(S)$ of D-modules on $S$ with singular support in $\cN\subset T^*S$ when $S$ is a smooth scheme of finite type. Now, fix a Zariski-closed conical subset $\cN\subset T^*\cY$.

For each smooth morphism $f\colon S\to\cY$ from a smooth scheme of finite type, let $\cN_S$ be the image of $\cN$ under the co-differential \[df^\bullet\colon T^*\cY\times_\cY S\to T^*S.\] Then we define the category $\Dmod_\cN(\cY)$ of D-modules with singular support in $\cN$ as
\[
\Dmod_\cN(\cY):=\lim_S\Dmod_{\cN_S}(S)
\]
where the limit is over finite-type smooth schemes and the transition maps are $!$-pullbacks. 
As in \cite[Definition~2.4.2]{HTT}, we define the notion of a non-characteristic morphism.
\begin{defn}\label{defn:non-char}
    Let $f\colon\mathcal X\to\mathcal Y$ be a morphism between smooth algebraic stacks, and let $\cN\subset T^*\cY$. We let
    \[
    T^*_{\mathcal X}\cY:=(df^\bullet)^{-1}(T_{\mathcal X}^*\mathcal X)\subset T^*\mathcal Y\times_{\mathcal Y}\mathcal X.
    \]
    Then $f$ is \emph{non-characteristic} with respect to $\cN$ if 
    \[
    (\cN\times_\cY\mathcal X)\cap T^*_{\mathcal X}\cY\subset T^*_\cY\cY\times_\cY\mathcal X.
    \]
\end{defn}
\begin{remark}
    Definition~\ref{defn:non-char} slightly differs from \cite[Definition~2.4.2]{HTT}; for them, a morphism $f$ is non-characteristic with respect to a coherent D-module $\cF$ on $\cY$ when $f$ is non-characteristic with respect to the singular support $\operatorname{SS}(\cF)\subset T^*\cY$ in our sense. We choose to define non-characteristic morphisms in terms of $\cN$ instead because we \emph{do not} define the singular support of a D-module; we only define what it means for a D-module $\cF$ to lie in the subcategory $\Dmod_\cN(\cY)$.
\end{remark}
We then have the following generalization of \cite[Lemma~2.2.5.1]{quasi-temp} to non-smooth morphisms.
\begin{lemma}\label{lem:SS-functorial}
    Let $f\colon\mathcal X\to\mathcal Y$ be a morphism between smooth algebraic stacks which is non-characteristic with respect to a Zariski-closed conical subset $\cN_\cY\subset T^*\cY$. Let $\cN_{\mathcal X}$ be the Zariski closure of the image of $\cN_\cY\times_\cY\mathcal X$ under the co-differential $df^\bullet$. Then the functor $f^!\colon\Dmod(\cY)\to\Dmod(\mathcal X)$ restricts to a functor
    \[
    f^!\colon\Dmod_{\cN_\cY}(\cY)\to\Dmod_{\cN_{\mathcal X}}(\mathcal X).
    \]
    Furthermore, in this situation, $f^!\simeq f^*[2\operatorname{dim.rel.}(f)]$, where the relative dimension $\operatorname{dim.rel.}(f)$ is considered a locally constant function on $\cY$.
\end{lemma}
\begin{proof}
    As in the proof of \cite[Lemma~2.2.5.1]{quasi-temp} this reduces to the analogous statement for smooth schemes of finite type. This is \cite[Theorem~2.4.6(iii)]{HTT} and \cite[Theorem~2.7.1]{HTT}.
\end{proof}

\subsection{Maximally anti-tempered D-modules}\label{subsec:joakim-result} 
Let $LG$ be the loop group of $G$ and let $L^+G$ be the positive loop group. Recall that the derived Satake equivalence describes the category
\[
\Sph_G:=\Dmod(L^+G\backslash LG/L^+G),
\]
which carries the convolution monoidal structure.
\begin{prop}[{\cite[\S12]{singular-support}, \cite[Theorem~5]{derived-satake}}]
    There is a monoidal equivalence of categories
    \[
    \Sat_G\colon\Sph_G\simeq\IndCoh_{\check\cN/\check G}(\Omega\check\fg/\check G)
    \]
    where $\check\cN\subset\check\fg$ is the nilpotent cone (viewed as a subspace of $\check\fg^*$ by an invariant form) and $\Omega\check\fg=0\times_{\check\fg}0$.
\end{prop}
Following \cite[\S2.3.3]{quasi-temp}, we define the \emph{quasi-tempered spherical category} as
\[
\Sph_G^{\operatorname{quasi-temp}}:=\Sat_G^{-1}\!\big(\IndCoh_{\check\cN_\irreg/\check G}(\Omega\check\fg/\check G)\big)
\]
where $\check\cN_\irreg=\check\cN\backslash\check\cN_\reg$ is the closed subscheme of $\check\cN$ consisting of irregular nilpotent elements. Define the \emph{maximally anti-tempered spherical category} as the quotient
\[\Sph_G^{\operatorname{max-anti-temp}}:=\Sph_G/\Sph_G^{\operatorname{quasi-temp}},\]
which sits in an exact sequence
\[\begin{tikzcd}
	{\Sph_G^{\operatorname{quasi-temp}}} && {\Sph_G} && {\Sph_G^{\operatorname{max-anti-temp}}}.
	\arrow[shift left=1, hook, from=1-1, to=1-3]
	\arrow["{\operatorname{quasi-temp}}", shift left=1, from=1-3, to=1-1]
	\arrow[shift left=1, hook, from=1-5, to=1-3]
	\arrow["{\operatorname{max-anti-temp}}", shift left=1, from=1-3, to=1-5]
\end{tikzcd}\]
For a fixed point $x\in X(e)$ there is a correspondence
\[
\begin{tikzcd}
    &\Hck_{G,x}\arrow[swap]{dl}{h^\leftarrow}\arrow{r}\arrow{dr}{h^\rightarrow}&\cHck_{G,x}\\
    \Bun_G&&\Bun_G,
\end{tikzcd}
\]
where $\Hck_{G,x}$ is the global Hecke stack and $\cHck_{G,x}=L_x^+G\backslash L_xG/L_x^+G$ is the local Hecke stack. This defines an action of $\Sph_{G,x}:=\Dmod(\cHck_{G,x})$ on $\Dmod(\Bun_G)$. Now let
\begin{align*}
    \Dmod(\Bun_G)^{\operatorname{quasi-temp}}&:=\Sph_{G,x}^{\operatorname{quasi-temp}}\otimes_{\Sph_{G,x}}\Dmod(\Bun_G)\\
    \Dmod(\Bun_G)^{\operatorname{max-anti-temp}}&:=\Sph_{G,x}^{\operatorname{max-anti-temp}}\otimes_{\Sph_{G,x}}\Dmod(\Bun_G),
\end{align*}
which also sits in an exact sequence
\[\begin{tikzcd}
	{\Dmod(\Bun_G)^{\operatorname{quasi-temp}}} && {\Dmod(\Bun_G)} && {\Dmod(\Bun_G)^{\operatorname{max-anti-temp}}}.
	\arrow[shift left=1, hook, from=1-1, to=1-3]
	\arrow["{\operatorname{quasi-temp}}", shift left=1, from=1-3, to=1-1]
	\arrow[shift left=1, hook, from=1-5, to=1-3]
	\arrow["{\operatorname{max-anti-temp}}", shift left=1, from=1-3, to=1-5]
\end{tikzcd}\]
On the spectral side, the category $\Dmod(\Bun_G)^{\qt}$ also admits a description in terms of coherent singular supports.
Recall from \cite[Corollary~10.4.7]{singular-support} that $\Sing(\LS_{\check G})$ classifies, for a classical affine test scheme $S$, tuples $(\cP,\nabla,A)$ where $(\cP,\nabla)\in\Map(X^\dR\times S,\B\check G)$ and $A\in H^0\Gamma(X^\dR\times S,\check\fg_\cP^*)$. The substack $\Nilp\subset\Sing(\LS_{\check G})$ classifies tuples $(\cP,\nabla,A)$ where $A$ is nilpotent.
\begin{defn}
    Let $\Nilp_\irreg\subset\Nilp$ be the closed substack classifying tuples $(\cP,\nabla,A)$ such that $A$ is irregular nilpotent.
\end{defn}
The category $\Dmod(\Bun_G)^{\qt}$ sits in a commutative diagram
\[
\begin{tikzcd}
    \Dmod(\Bun_G)^{\qt}\arrow{rr}{\mathbb L_G^{\qt}}\arrow[hook]{d}&&\IndCoh_{\Nilp_\irreg}(\LS_{\check G})\arrow[hook]{d}\\
    \Dmod(\Bun_G)\arrow{rr}{\mathbb L_G}&&\IndCoh_\Nilp(\LS_{\check G}).
\end{tikzcd}
\]

The main theorem of \cite{quasi-temp} gives a description of the maximally anti-tempered category purely in terms of D-modules.
\begin{defn}[{\cite[\S1.2.1]{quasi-temp}}]
    Let $\Dmod_{\fz(\fg)}(\Bun_G)$ be the subcategory of $\Dmod(\Bun_G)$ consisting of objects $\cF$ such that if the Higgs bundle $(\cP,\phi)\in T^*\Bun_G$ lies in the singular support of $\cF$, then $\phi$ factors through the center $\fz(\fg)\subset\fg$.
\end{defn}
\begin{thm}[{\cite[Theorem~A]{quasi-temp}}]\label{thm:joakim}
     $\Dmod(\Bun_G)^{\mat}=\Dmod_{\fz(\fg)}(\Bun_G)$.
\end{thm}

\section{Compatibility of Geometric Langlands and isogenies}\label{sec:compatibility-of-geometric-langlands-and-isogenies}
To describe the image of arbitrary finite-rank local systems on $\Bun_G$---not just the constant sheaf---we will need to analyze $\LS_{\check Z}$ where $\check Z\subset\check G$ is the center. In particular, we seek a geometric Langlands equivalence for the possibly disconnected torus $\check Z$, which we develop in this section. Readers who are willing to assume the group $\check G$ has connected center are welcome to skip this section.
\subsection{Recollections for connected tori}
Let $T$ be a torus and let $\check T$ be its Langlands dual. The Fourier--Mukai equivalence is an equivalence
\begin{equation}\label{eq:usual-FM}
\mathrm{FM}\colon\QCoh(\Bun_T)\simeq\QCoh(\Bun_{\check T}),
\end{equation}
with the Poincar\'e line bundle $\cL_{\Poinc}\in\QCoh(\Bun_T\times\Bun_{\check T})$ as a kernel. 
\begin{example}\label{example:image-under-FM}
    The structure sheaf $\cO_{\Bun_T}\in\QCoh(\Bun_T)$ is sent to the skyscraper sheaf $\delta_{\triv_{\check T}*}[-\delta_T]$ where $\delta_T=\dim(\Bun_T)$. Conversely, the skyscraper sheaf $\delta_{\triv_T*}$ is sent to $\cO_{\Bun_{\check T}}$.
\end{example}
By \cite{laumon-fourier,rothstein} the equivalence~\eqref{eq:usual-FM} can be enhanced to an equivalence
\begin{equation}\label{eq:FM-enhanced}
\mathrm{FM}^{\mathrm{enh}}\colon\Dmod(\Bun_T)\simeq\QCoh(\LS_{\check T}),
\end{equation}
which fits into a commutative diagram \cite[\S1.5.2]{GLC1}
\[
\begin{tikzcd}
    \Dmod(\Bun_T)\arrow{r}{\mathrm{FM}^{\mathrm{enh}}}\arrow[swap]{d}{\oblv^r}&\QCoh(\LS_{\check T})\arrow{d}{p_*}\\
    \QCoh(\Bun_T)\arrow{r}{\mathrm{FM}}&\QCoh(\Bun_{\check T}),
\end{tikzcd}
\]
where $p\colon\LS_{\check T}\to\Bun_{\check T}$ is the morphism forgetting the connection. Then the geometric Langlands equivalence is
\begin{equation}\label{eq:GLC-tori-defn}
\mathbb L_T:=\mathrm{FM}^{\mathrm{enh}}\circ\tau_T
\end{equation}
where $\tau_T$ is negation on $\Bun_T$. The equivalence is a monoidal equivalence, where:
\begin{itemize}
    \item $\Dmod(\Bun_T)$ carries the usual tensor product $\cF\otimes^!\cG$; and
    \item $\QCoh(\LS_{\check T})$ carries the convolution product, defined by $\cF\star\cG:=m_*(\cF\boxtimes\cG)$ where $m$ is the morphism $\LS_{\check T}\times\LS_{\check T}\to\LS_{\check T}$ induced from multiplication.
\end{itemize}
For our purposes, it is convenient to extend the equivalence to the case when $\check T$ is possibly disconnected (but still commutative). 
\begin{example}\label{ex:constant-D-module-under-GLC-tori}
    The constant D-module $\omega_{\Bun_T}\in\Dmod(\Bun_T)$ is sent to $\delta_{\triv_{\check T}*}\in\QCoh(\LS_{\check T})$. These are the monoidal units in the monoidal structures defined above. This is compatible with Example~\ref{example:image-under-FM}, since $\oblv^r(\omega_{\Bun_T})=\cO_{\Bun_T}[\delta_T]$.
\end{example}

\subsection{Disconnected tori}\label{subsec:disconnected-tori}
We will extend the geometric Langlands equivalence to disconnected tori. First, we need to make sense of the ``Langlands dual group" of a disconnected torus.

Recall that for a connected torus $\check T$, if $\Lambda=\Hom(\check T,\bG_m)$ is the character lattice, then $\check T=\Spec(e[\Lambda])$. The Langlands dual torus is defined as $T:=\Spec(e[\Lambda^\vee])$ where $\Lambda^\vee=\Hom(\Lambda,\Z)$. We now extend this definition to the case where $\check T$ is not necessarily connected.

\begin{defn}
    A \emph{possibly disconnected torus} is a commutative affine algebraic group $\check T$ over $e$ whose identity component $\check T^\circ$ is a torus. The \emph{character lattice} $X^*(\check T)$ is the group of homomorphisms $\check T\to\bG_m$. 
    The \emph{Langlands dual group} of $\check T$ is the Picard stack
    \begin{equation}\label{def:Langlands-dual}
    T:=\bG_m\otimes^{\mathbb L}_{\mathbb Z} X^*(\check T).
    \end{equation}
\end{defn}
\begin{remark}
    The derived tensor in \eqref{def:Langlands-dual} can be computed explicitly by presenting $X^*(\check T)$ as $\Lambda_1/\Lambda_2$ with free abelian groups $\Lambda_i$ and letting $T$ be the stacky quotient $(\bG_m\otimes\Lambda_1)/(\bG_m\otimes\Lambda_2)$.
\end{remark}
\begin{example}
    The character lattice of $\check T=\mu_n$ is $X^*(\mu_n)=\Z/n\Z$, so its Langlands dual is
    \[
    T=\bG_m\otimes^{\mathbb L}_{\Z}\Z/n\Z\simeq\B\mu_n.
    \]
\end{example}
For us, the main example is the following.
\begin{defn}
    The \emph{stacky abelianization}\footnote{The notation $G_{\abc}$ stands for \emph{\underline{ab}elianization (\underline{c}orrected)}.} of $G$ is $G_{\abc}:=G/[G,G]^{\mathrm{sc}}$ where $[G,G]^{\mathrm{sc}}$ is the simply connected cover of the derived subgroup $[G,G]\subset G$.
\end{defn}
\begin{remark}
    When $[G,G]$ is simply connected, this matches the usual notion of abelianization. However, generally the automorphism group of a point in $G_{\abc}$ is $\pi_1([G,G])$.
\end{remark}
\begin{lemma}\label{lem:dual-of-Z}
    Let $G$ be a reductive group, and let $\check Z$ be the center of the Langlands dual $\check G$. Then the dual of $\check Z$ is $G_\abc$.
\end{lemma}
\begin{proof}
Let $\check T\subset\check G$ be a maximal torus and let $\check T_{\ad}\subset\check G_{\ad}$ be the image in the adjoint group; dually, let $T\subset G$ be a maximal torus and let $T^{\mathrm{sc}}\subset[G,G]^{\mathrm{sc}}$ be the preimage of $T\cap[G,G]$. Then since the dual of the adjoint group $\check G/\check Z$ is $[G,G]^{\mathrm{sc}}$, the dual of $\check Z=\ker(\check T\to\check T_{\ad})$ is $T/T^{\mathrm{sc}}$. Moreover, the natural map $T/T^{\mathrm{sc}}\to G/[G,G]^{\mathrm{sc}}=G_{\abc}$ is an isomorphism.
\end{proof}
In fact, by the structure theorem of finitely generated abelian groups, we may classify all possibly disconnected tori.
\begin{lemma}\label{lem:structure-thm-disconnected-tori}
    Any disconnected torus $\check T$ is isomorphic to $\check T^\circ\times\Gamma$ for some connected torus $\check T^\circ$ and a finite abelian group scheme $\Gamma$. Its Langlands dual is isomorphic to $T^\circ\times\B\Gamma^\vee(1)$ where $T^\circ$ is the Langlands dual of $\check T^\circ$ and $\Gamma^\vee(1)$ is the Tate twist of the Cartier dual of $\Gamma$.
\end{lemma}

Now, the geometric Langlands equivalence for disconnected tori states
\begin{prop}\label{prop:GLC-disconnected}
    Let $\check T$ be a possibly disconnected torus and let $T$ be the Langlands dual group. There is a canonical equivalence
    \begin{equation}\label{eq:disconnected-tori-equivalence}
    \mathbb L_T\colon \Dmod(\Bun_T)\simeq\QCoh(\LS_{\check T}).
    \end{equation}
\end{prop}
\begin{proof}
We construct the enhanced Fourier--Mukai equivalence $\mathrm{FM}^{\mathrm{enh}}$ as in \eqref{eq:FM-enhanced}, and compose with $\tau_T$. We first construct a Poincar\'e line bundle on $\Bun_T^\dR\times\LS_{\check T}$. Using Lemma~\ref{lem:structure-thm-disconnected-tori} we choose an isomorphism $\check T\simeq\check T^\circ\times\Gamma$ and $T\simeq T^\circ\times\B\Gamma^\vee(1)$ where $\check T^\circ$ is a connected torus and $\Gamma$ is a finite abelian group scheme. Now
\[
\Bun_{T^\circ}^\dR\times\LS_{\check T^\circ}
\]
carries the usual Poincar\'e line bundle considered by \cite{laumon-fourier,rothstein} which induces an equivalence. On the other hand
\[
\Bun_{\B\Gamma^\vee(1)}^\dR\times\LS_\Gamma\simeq\Ge_{\Gamma^\vee(1)}\times\Bun_\Gamma
\]
carries a Verdier duality pairing to $\B\bG_m$, as described in \cite[\S8.2.1]{GLC5}. As in \cite[\S8.2.4]{GLC5}, a choice of base point $x\in X(e)$ gives splittings
\[
\Ge_{\Gamma^\vee(1)}\simeq\B^2\Gamma^\vee(1)\times \B H^1_\et(X,\Gamma^\vee(1))\times H^2_\et(X,\Gamma^\vee(1)),\ \Bun_\Gamma\simeq\B\Gamma\times H^1_\et(X,\Gamma).
\]
The pairing is trivial on $\B^2\Gamma^\vee(1)$ and is given by the perfect pairings
\[ H^1_\et(X,\Gamma^\vee(1))\times H^1_\et(X,\Gamma)\to\bG_m,\ H^2_\et(X,\Gamma^\vee(1))\times \Gamma\to\bG_m.\]
Since $\QCoh(\B^2\Gamma^\vee(1))\simeq\Vect$, the 2-stacky part of $\Ge_{\Gamma^\vee(1)}$ does not contribute and hence the associated Fourier--Mukai functor
\[
\Dmod(\Bun_{\B\Gamma^\vee(1)})\to\QCoh(\LS_\Gamma)
\]
is an equivalence. Together we obtain an equivalence
\[
\Dmod(\Bun_T)\simeq\QCoh(\LS_{\check T}).
\]
Furthermore we may readily check the equivalence is independent of the splitting $\check T\simeq\check T^\circ\times\Gamma$.
\end{proof}

\subsection{Functoriality} Now let $\check\phi\colon\check T_1\to\check T_2$ be an arbitrary homomorphism between possibly disconnected tori. There is an induced homomorphism $\phi\colon T_2\to T_1$ between their Langlands duals. By an abuse of notation we also denote the induced morphisms $\LS_{\check T_1}\to\LS_{\check T_2}$ and $\Bun_{T_2}\to\Bun_{T_1}$ by $\check\phi$ and $\phi$, respectively.
\begin{lemma}\label{lem:GLC-tori-functorial}
    There is a commutative diagram
    \[
    \begin{tikzcd}
        \Dmod(\Bun_{T_1})\arrow{r}{\mathbb L_{T_1}}\arrow[swap]{d}{\phi^!}&\QCoh(\LS_{\check T_1})\arrow{d}{\check\phi_*}\\
        \Dmod(\Bun_{T_2})\arrow{r}{\mathbb L_{T_2}}&\QCoh(\LS_{\check T_2}).
    \end{tikzcd}
    \]
\end{lemma}
\begin{cor}\label{cor:GLC-tori-monoidal-equivalence}
    Let $\check T$ be a possibly disconnected torus and $T$ be the Langlands dual group. The equivalence~\eqref{eq:disconnected-tori-equivalence} is monoidal under the usual monoidal structure on $\Dmod(\Bun_T)$ and the convolution product on $\QCoh(\LS_{\check T})$.
\end{cor}
\begin{proof}
    By Lemma~\ref{lem:GLC-tori-functorial} applied to the multiplication homomorphism $m\colon\check T\times\check T\to\check T$ whose dual is the diagonal morphism $\Delta\colon T\to T\times T$.
\end{proof}

\subsection{Compatibility of actions}
Recall the following from \cite[Proposition~2.1.6]{GLC5}.
\begin{prop}\label{prop:compatibility-with-isogeny}
    Let $G$ and $G'$ be reductive groups and let $\phi\colon G\to G'$ be an almost isogeny, i.e., a homomorphism inducing an isogeny on derived subgroups. Denote the dual homomorphism as $\check\phi\colon\check G'\to\check G$. By an abuse of notation let $\phi$ and $\check\phi$ also denote the induced morphisms $\Bun_G\to\Bun_{G'}$ and $\LS_{\check G'}\to\LS_{\check G}$. There is a commutative diagram
    \[
    \begin{tikzcd}
        \Dmod(\Bun_{G'})\arrow{r}{\mathbb L_{G'}}\arrow{d}{\phi^!}&\IndCoh_\Nilp(\LS_{\check G'})\arrow{d}{\check\phi_*^{\IndCoh}}\\
        \Dmod(\Bun_{G})\arrow{r}{\mathbb L_G}&\IndCoh_\Nilp(\LS_{\check G}).
    \end{tikzcd}
    \]
\end{prop}
\begin{remark}
    Compare with Lemma~\ref{lem:GLC-tori-functorial}.
\end{remark}
As a consequence of Proposition~\ref{prop:compatibility-with-isogeny}, we conclude the following.
\begin{thm}\label{thm:GLC-intertwines-actions}
    Let $G$ be a reductive group. Then the geometric Langlands equivalence $\mathbb L_G\colon\Dmod(\Bun_G)\simeq\IndCoh_\Nilp(\LS_{\check G})$ intertwines the following actions under the monoidal equivalence $\Dmod(\Bun_{G_{\abc}})\simeq\QCoh(\LS_{\check Z})$ of Proposition~\ref{prop:GLC-disconnected} and Corollary~\ref{cor:GLC-tori-monoidal-equivalence}:
    \begin{enumerate}
        \item $\cF\in\Dmod(\Bun_{G_{\abc}})$ acts on $\Dmod(\Bun_G)$ by $\cG\mapsto\pi_{\Bun_G}^!\cF\otimes^!\cG$; and
        \item $\cF\in\QCoh(\LS_{\check Z})$ acts on $\IndCoh_\Nilp(\LS_{\check G})$ via convolution $\cG\mapsto m_*(\cF\boxtimes\cG)$ where $m\colon\LS_{\check Z}\times\LS_{\check G}\to\LS_{\check G}$ is induced by the multiplication homomorphism $m\colon\check Z\times\check G\to\check G$. 
    \end{enumerate}
\end{thm}
\begin{proof}
    Choose an isomorphism $\check Z=\check Z^\circ\times\Gamma$ with dual group $Z\simeq G_\ab\times\B\Gamma^\vee(1)$. Now, it suffices to check the compatibility of the actions of $\Dmod(\Bun_{G_\ab})\simeq\QCoh(\LS_{\check Z^\circ})$ and $\Dmod(\Ge_{\Gamma^\vee(1)})\simeq\QCoh(\LS_\Gamma)$. The former follows from Proposition~\ref{prop:compatibility-with-isogeny} applied to the homomorphism $\Delta\colon G\to G_\ab\times G$ with dual $m\colon\check Z\times \check G\to\check G$:
    \[
    \begin{tikzcd}
        \Dmod(\Bun_{G_\ab})\otimes\Dmod(\Bun_G)\arrow{rr}{\mathbb L_{G_\ab}\otimes\mathbb L_G}\arrow{d}{\Delta^!}&&\QCoh(\LS_{\check Z^\circ})\otimes\IndCoh_\Nilp(\LS_{\check G})\arrow{d}{m_*^{\IndCoh}}\\
        \Dmod(\Bun_G)\arrow{rr}{\mathbb L_G}&&\IndCoh_\Nilp(\LS_{\check G}).
    \end{tikzcd}
    \]
    The latter follows from \cite[\S8.5]{GLC5}.
\end{proof}

\section{Lafforgue's expectation}\label{sec:lafforgue-expectation}
We will formally state Lafforgue's expectation in this section. We will use the following explicit description of $\LS_{\bG_a}$.
\begin{lemma}\label{lem:LS_Ga}
    There is an isomorphism $\LS_{\bG_a}\simeq\cC^\bullet_\dR(X)[1]$.\footnote{Here, by an abuse of notation we implicitly view the complex $\cC^\bullet_\dR(X)[1]$ as a derived stack; this is the \emph{linear stack} $\mathbb V(\cC^\bullet_\dR(X)[1])$ as defined in \cite[\S3.2]{toen}.}
    In particular, non-canonically, \[\LS_{\bG_a}\simeq\Omega\bA^1\times H^1_\dR(X)\times\B\bG_a.\]
\end{lemma}
\begin{proof}
    An $S$-point of $\LS_{\bG_a}$ is equivalent to a symmetric monoidal functor
    \[
    \Rep(\bG_a)\to \QCoh(S)\otimes\Dmod(X).
    \]
    Since $\Rep(\bG_a)\simeq e\langle\xi\rangle\operatorname{-mod}$ where $\xi$ has degree one, such a functor amounts to fixing a morphism
    \[
    \cO_S\boxtimes\omega_X\to\cO_S\boxtimes\omega_X[1],
    \]
    which is equivalently a morphism $\cO_S\to \cO_S\otimes\cC_\dR^\bullet(X)[1]$,
    as desired. 
    Since the chain complex $\cC^\bullet_\dR(X)$ is formal, there is a non-canonical splitting.
\end{proof}
Note the following description of the category of finite-rank local systems on $\Bun_G$. Following \cite[\S4.1.1]{quasi-temp}, for a prestack $\cY$ we denote by $\Dmod(\cY)^{\mathrm{coh}}$ the category of locally coherent D-modules, i.e., D-modules $\cF$ such that for any smooth map $f\colon S\to\cY$ from an affine scheme of finite type, the pullback $f^!\cF\in\Dmod(S)$ is a bounded complex with coherent cohomologies.
\begin{prop}\label{prop:fin-rank-local-systems}
    Let $\pi_{\Bun_G}$ be the obvious morphism $\Bun_G\to\Bun_{G_\abc}$. Then
    \[
    \pi_{\Bun_G}^![\delta_Z-\delta_G]^{\heartsuit,\mathrm{coh}}\colon\Dmod_0(\Bun_{G_\abc})^{\heartsuit,\mathrm{coh}}\to\Dmod_0(\Bun_G)^{\heartsuit,\mathrm{coh}}
    \]
    is an equivalence. 
\end{prop}
\begin{proof}
    When $G$ is simply connected, the left-hand side is simply the abelian category of finite-dimensional vector spaces, and this was proved in \cite[Proposition~4.1.4.1]{quasi-temp}. 

For general $G$, this follows from the same strategy, by observing that on the underlying topological stacks,
    \[
    \Bun_G^{\mathrm{top}}\to\Bun_{G_\abc}^{\mathrm{top}}
    \]
    is a homotopy equivalence on the $\tau_{\le1}$ truncations. For semisimple $G$ this was proved in \cite[Proposition~4.1.6]{GLC5} and the same strategy shows this for general $G$.
\end{proof}
Thus, to describe all finite-rank local systems on $\Bun_G$, it suffices to characterize the functor $\pi^!_{\Bun_G}\colon\Dmod(\Bun_{G_\abc})\to\Dmod(\Bun_G)$. In fact, the functor factors through $\Dmod_{\fz(\fg)}(\Bun_G)$, the category which is the subject of Theorem~\ref{thm:joakim}.\footnote{This is by the functoriality of singular support, and observing that the perpendicular of $[\fg,\fg]\subset\fg$ under a non-degenerate invariant form is $\fz(\fg)\subset\fg$.} We will extend Lafforgue's conjecture to this situation.

First, we will replace the stack $\LS_{\check N}/\check\rho(\bG_m)$ in \eqref{eq:intro-corr} by a stack $\LS_{\check Z\check N}/\check\rho(\bG_m)$. 
Consider the extension
\begin{equation}\label{eq:SES-of-Z}
1\to\check Z\to\check Z_{\mathrm{enh}}:=\frac{\check Z\times\bG_m}{\mu_2}\to \bG_m\to 1,
\end{equation}
where
\begin{itemize}
    \item $\mu_2\hookrightarrow\check Z\times\bG_m$ is the diagonal embedding $\epsilon\mapsto(2\check\rho(\epsilon),\epsilon)$;
\item $\check Z\to \check Z_{\mathrm{enh}}$ is $a\mapsto (a,1)$; and
\item $\check Z_{\mathrm{enh}}\to\bG_m$ is $(a,t)\mapsto t^2$.
\end{itemize}
Let us form the semidirect product
\[
\check Z_{\mathrm{enh}}\ltimes_{2\check\rho}\check N,
\]
where $(a,t)\in\check Z_{\mathrm{enh}}$ acts on $\check N$ via $\mathrm{ad}(2\check\rho(t))$. Then \eqref{eq:SES-of-Z} gives rise to an extension
\[
1\to\check Z\check N\to\check Z_{\mathrm{enh}}\ltimes_{2\check\rho}\check N\to \bG_m\to 1.
\]
This defines an action of $\bG_m$ on $\B(\check Z\check N)=\B(\check Z_{\mathrm{enh}}\ltimes_{2\check\rho}\check N)\times_{\B\bG_m}1$ which further defines an action of $\bG_m$ on $\LS_{\check Z\check N}=\Map(X^\dR,\B(\check Z\check N))$. We denote the quotient by $\LS_{\check Z\check N}/\check\rho(\bG_m)$.

Now, we consider the correspondence
\begin{equation}\label{eq:lafforgue-correspondence}
\begin{tikzcd}
    &\LS_{\check Z\check N}/\check\rho(\bG_m)\arrow[swap]{dl}{\iota}\arrow{dr}{\Sigma}\\
    \LS_{\check G}&&(\LS_{\check Z}\times\Omega\bA^1)/\bG_m,
\end{tikzcd}\end{equation}
where the action of $\bG_m$ on $\LS_{\check Z}\times\Omega\bA^1$ is diagonal: on the first factor it is via the extension $\check Z_{\mathrm{enh}}$ as in \eqref{eq:SES-of-Z} and on the second factor it is of weight one.

The morphism $\iota$ is defined as the composition
\[
\LS_{\check Z\check N}/\check\rho(\bG_m)\to\LS_{\check Z_{\mathrm{enh}}\ltimes_{2\check\rho}\check N}\to\LS_{\check G},
\]
where the second morphism is induced from the homomorphism
\[\check Z_{\mathrm{enh}}\ltimes_{2\check\rho}\check N\to\check G:(a,t,n)\mapsto a\cdot (2\check\rho)(t)\cdot n.\]
The morphism $\Sigma$ is defined via the obvious homomorphism $\check Z\check N\to\check Z$ and a non-degenerate character $\check N\to\bG_a$. 

The following functor is a generalization of \eqref{intro-eq:defn-of-poinc-spec}.

\begin{defn}
    Using notation from \eqref{eq:singularity-functor2}, let $\Poinc^\spec$ be the composition
\begin{equation}\label{eq:spectral-Poinc-defn}
\QCoh(\LS_{\check Z})\xrightarrow{\exp^\spec_{\LS_{\check Z}}}\IndCoh_{0\times\bA^1}((\LS_{\check Z}\times\Omega\bA^1)/\bG_m)\xrightarrow{\iota^{\IndCoh}_*\Sigma^!}\IndCoh_\Nilp(\LS_{\check G}).
\end{equation}
\end{defn}
\begin{remark}
    The object $\Poinc^\spec$ as defined in the introduction \eqref{intro-eq:defn-of-poinc-spec} is $\Poinc^\spec(\delta_{\triv_{\check Z}*})$.
\end{remark}
Lafforgue's conjecture may now be stated.
\begin{thm}\label{conj:lafforgue-Z}
    Let $\pi_{\Bun_G}\colon\Bun_G\to\Bun_{G_\abc}$ be the standard morphism. There is a commutative diagram
    \begin{equation}\label{eq:main-diagram}
    \begin{tikzcd}
    \Dmod(\Bun_{G_{\mathrm{abc}}})\arrow{r}{\mathbb L_{G_{\mathrm{abc}}}}\arrow{d}{\pi_{\Bun_G}^!}&\QCoh(\LS_{\check Z})\arrow{d}{\Poinc^\spec}\\
        \Dmod(\Bun_G)\arrow{r}{\mathbb L_G}&\IndCoh_\Nilp(\LS_{\check G})
    \end{tikzcd}
    \end{equation}
    where $G_{\mathrm{abc}}=G/[G,G]^{\mathrm{sc}}$ and $\bL_{G_\abc}$ is the equivalence of Proposition~\ref{prop:GLC-disconnected} (we use Lemma~\ref{lem:dual-of-Z} to identify the dual of $\check Z$ with $G_\abc$).
\end{thm}
We prove Theorem~\ref{conj:lafforgue-Z} in \S\ref{subsec:main-proof}. The cohomological shifts will be worked out in the proof of Theorem~\ref{prop:image-of-constant-sheaf}---everything happens to cancel out!

We may dualize everything. Denote by $\Whit^\spec$ the composition
\begin{equation}\label{eq:defn-of-whit^spec}
\IndCoh_\Nilp(\LS_{\check G})\xrightarrow{\Sigma^{\IndCoh}_*\iota^!}\IndCoh_{0\times\bA^1}((\LS_{\check Z}\times\Omega\bA^1)/\bG_m)\xrightarrow{\sing_{\LS_{\check Z}}}\QCoh(\LS_{\check Z}).
\end{equation}
\begin{thm}\label{thm:laf1}
    There is a commutative diagram
    \[
    \begin{tikzcd}
        \Dmod(\Bun_G)\arrow{r}{\mathbb L_G}\arrow{d}{\pi_{\Bun_G!}}&\IndCoh_\Nilp(\LS_{\check G})\arrow{d}{\Whit^\spec[2(\delta_G+\delta_N-\delta_Z)]}\\
        \Dmod(\Bun_{G_{\mathrm{abc}}})\arrow{r}{\mathbb L_{G_{\mathrm{abc}}}}&\QCoh(\LS_{\check Z}),
    \end{tikzcd}
    \]
    where $G_{\mathrm{abc}}=G/[G,G]^{\mathrm{sc}}$ and $\bL_{G_\abc}$ is the equivalence of Proposition~\ref{prop:GLC-disconnected}.
\end{thm}
Lafforgue's two conjectures, Theorem~\ref{thm:laf1} and Theorem~\ref{conj:lafforgue-Z}, are equivalent by the dualities of \S\ref{subsec:duality}.

\begin{proof}[Proof that Theorem~\ref{conj:lafforgue-Z} implies Theorem~\ref{thm:laf1}]
First of all, the miraculous duality functors for $G$ and $G_\abc$ sit in a commutative diagram
\begin{equation}\label{eq:miraculous-dual}
\begin{tikzcd}
\Dmod(\Bun_G)^\vee\arrow{r}{\mathbb D_!}\arrow{d}{(\pi^*_{\Bun_G})^\vee}&\Dmod(\Bun_G)\arrow{d}{\pi_{\Bun_G!}}\\
\Dmod(\Bun_{G_\abc})^\vee\arrow{r}{\mathbb D_!}&\Dmod(\Bun_{G_\abc})
\end{tikzcd}
\end{equation}
Indeed, by base change
\[
(\pi_{\Bun_G}\times1)^*\Delta_{\Bun_{G_\abc}!}e_{\Bun_{G_\abc}}=(1\times\pi_{\Bun_G})_!\Delta_{\Bun_G!}(e_{\Bun_G})\in\Dmod(\Bun_G\times\Bun_{G_\abc}),
\]
so for any $F\in\Dmod(\Bun_G)^\vee$,
\begin{equation}
(F\otimes\operatorname{Id})(\pi_{\Bun_G}\times1)^*\Delta_{\Bun_{G_\abc}!}(e_{\Bun_{G_\abc}})=(1\times\pi_{\Bun_G})_!(F\otimes\operatorname{Id})\Delta_{\Bun_G!}e_{\Bun_G}.
\end{equation}
Now there is a commutative diagram
\[
\begin{tikzcd}
    \Dmod(\Bun_G)\arrow[swap,bend left=10]{rrr}{\tau[2(\delta_G+\delta_N)]\circ\mathbb L_G}\arrow{d}{\pi_{\Bun_G!}}&\arrow{l}{\mathbb D_![2\delta_G]}\Dmod(\Bun_G)^\vee\arrow{d}{(\pi_{\Bun_G}^!)^\vee}&\arrow{l}{\mathbb L_G^\vee}\IndCoh_\Nilp(\LS_{\check G})^\vee\arrow{d}{(\Poinc^\spec)^\vee}\arrow[swap]{r}{\mathbb D_{\mathrm{Serre}}}&\IndCoh_\Nilp(\LS_{\check G})\arrow{d}{\Whit^\spec}\\
    \Dmod(\Bun_{G_{\mathrm{abc}}})\arrow[bend right=10]{rrr}{\tau[2\delta_Z]\circ\mathbb L_{G_\abc}}&\arrow[swap]{l}{\mathbb D_![\delta_Z]}\Dmod(\Bun_{G_{\mathrm{abc}}})^\vee\arrow{r}{\mathbb L_{G_{\mathrm{abc}}}^\vee}&\arrow[swap]{l}{\mathbb L_{G_\abc}^\vee}\QCoh(\LS_{\check Z})^\vee\arrow{r}{\mathbb D_{\mathrm{Serre}}}&\QCoh(\LS_{\check Z}).
    \end{tikzcd}
\]
Here, the left square is \eqref{eq:miraculous-dual}, the middle square is by assumption, the right square is by Lemma~\ref{lem:serre-duality}, and the top and bottom are by Proposition~\ref{lem:duality-intertwined}. The large commutative square is the desired statement.
\end{proof}

\subsection{Computing the singular support}
As a basic sanity check for why Theorem~\ref{conj:lafforgue-Z} is reasonable, we have the following.

\begin{prop}\label{prop:lafforgue-functor-anti-temp}
    Lafforgue's functor $\Whit^\spec$ as defined in \eqref{eq:defn-of-whit^spec} vanishes on the subcategory $\IndCoh_{\Nilp_\irreg}(\LS_{\check G})$ of $\IndCoh_{\Nilp}(\LS_{\check G})$.
\end{prop}
\begin{proof}
We analyze the behavior of singular supports under the correspondence~\eqref{eq:lafforgue-correspondence}; in particular, under the functors $\iota^!$ and $\Sigma_*^{\IndCoh}$. 

Because $\LS_{\check Z\check N}\to\LS_{\check Z\check N}/\check\rho(\bG_m)$ is smooth, \cite[\S8.3.1]{singular-support} implies the singular co-differential is an isomorphism
\[
\Sing(\LS_{\check Z\check N}/\check\rho(\bG_m))\times_{\LS_{\check Z\check N}/\check\rho(\bG_m)}\LS_{\check Z\check N}\simeq\Sing(\LS_{\check Z\check N}).
\]
Thus, the singular co-differential of $\iota$ over an $S$-point $(\cP,\nabla)$ of $\LS_{\check Z\check N}$ is the obvious homomorphism
\[
\Sing(\iota)_{(\cP,\nabla)}\colon H^0(\Gamma(X^\dR\times S,\check\fg^*_\cP))\to H^0(\Gamma(X^\dR\times S,\Lie(\check Z\check N)^*_\cP)).
\]
Similarly, the singular co-differential of $\Sigma$ over an $S$-point $(\cP,\nabla)$ of $\LS_{\check Z\check N}$ is
\[
\Sing(\Sigma)_{(\cP,\nabla)}\colon H^0(\Gamma(X^\dR\times S,\Lie(\check Z)^*_\cP))\oplus\cO_S\to H^0(\Gamma(X^\dR\times S,\Lie(\check Z\check N)^*_\cP)).
\]
where the morphism from $\cO_S$ is induced from the non-degenerate character $\check N\to\bG_a$.

The Lie-theoretic fact that if $f\in\check\fn^-$ is a principal nilpotent then $(\bA^1\cdot f+\check\fb)\cap\check\cN_\irreg=\check\cN_\irreg\cap\check\fb$ (because all elements in $f+\check\fb$ are regular) implies that
\[
\Sing(\Sigma)_{(\cP,\nabla)}^{-1}\big(\Sing(\iota)_{(\cP,\nabla)}(\Nilp_\irreg)\big)=0.
\]
Thus, by the functoriality of singular support under pullback and pushforward as in \cite[Theorem~1.3.11]{singular-support}, the functor $\Sigma_*^{\IndCoh}\iota^!$ out of $\IndCoh_{\Nilp_\irreg}\!(\LS_{\check G})$ factors through
\[\IndCoh_0\big((\LS_{\check Z}\times\Omega\bA^1)/\bG_m\big)=\QCoh\!\big((\LS_{\check Z}\times\Omega\bA^1)/\bG_m\big).\] 
Since the singularity functor $\sing$ kills the subcategory
\[\QCoh\!\big((\LS_{\check Z}\times\Omega\bA^1)/\bG_m\big)\subset\IndCoh_{0\times\bA^1}\!\big((\LS_{\check Z}\times\Omega\bA^1)/\bG_m\big),\] we conclude that $\Whit^\spec$ kills $\IndCoh_{\Nilp_\irreg}(\LS_{\check G})$.
\end{proof}

An immediate consequence of Proposition~\ref{prop:lafforgue-functor-anti-temp} and Theorem~\ref{thm:joakim} is the following.
\begin{cor}\label{cor:poincare-factors}
    The functor $\mathbb L_G^{-1}\circ\Poinc^\spec$, where $\Poinc^\spec$ is as in Theorem~\ref{conj:lafforgue-Z}, factors through
    \[
    \Dmod_{\fz(\fg)}(\Bun_G)\subset\Dmod(\Bun_G).
    \]
\end{cor}
We also record the following further consequence.
\begin{cor}\label{cor:poincare-factors2}
    The object $\mathbb L_G^{-1}\circ\Poinc^\spec(\delta_{\triv_{\check Z}*})$ lies in the subcategory
    \[
    \Dmod_0(\Bun_G)\subset\Dmod(\Bun_G).
    \]
\end{cor}
\begin{proof}
    Corollary~\ref{cor:poincare-factors} already tells us the object lies in the subcategory $\Dmod_{\fz(\fg)}(\Bun_G)$. It remains to check the object also lands in the subcategory $\Dmod_\Nilp(\Bun_G)$ of D-modules with nilpotent singular support. By \cite[Proposition~14.5.3]{AGKRRV} it suffices to check that $\Poinc^\spec\in\IndCoh_\Nilp(\LS_{\check G}^{\mathrm{restr}})$ where $\LS_{\check G}^{\mathrm{restr}}$ is the stack of $\check G$-local systems with restricted variation introduced in \cite[\S1.4]{AGKRRV}.

This follows from observing that $\Poinc^\spec(\delta_{\triv_{\check Z}*})$ is defined by push-pull along the diagram
\[
\LS_{\check G}\leftarrow(\LS_{\mu_2}\times\LS_{\check N})/\check\rho(\bG_m)\to (\LS_{\mu_2}\times\Omega\bA^1)/\bG_m
\]
and the observation that the morphism $(\LS_{\mu_2}\times\LS_{\check N})/\check\rho(\bG_m)\to\LS_{\check G}$ factors through $\LS_{\check G}^{\mathrm{restr}}$, since 
\begin{equation}\label{eq:iso-of-restricted-stack}\LS_{\check N}^{\mathrm{restr}}\simeq\LS_{\check N}.\end{equation}
Indeed, any monoidal functor $\Rep(\check N)\to\QCoh(S)\otimes\Dmod(X)$ factors through $\QCoh(S)\otimes\operatorname{QLisse}(X)$, since $\Rep(\check N)$ is generated under colimits by the trivial representation, which must be sent to $\cO_S\boxtimes\omega_X\in\QCoh(S)\otimes\operatorname{QLisse}(X)$ (when $\check N=\bG_a$ the isomorphism \eqref{eq:iso-of-restricted-stack} is observed in \cite[\S1.5.2]{AGKRRV}).
\end{proof}

\section{Computation of constant term functors}\label{sec:some-computations}
In Corollary~\ref{cor:poincare-factors2} we constructed an object $\Poinc^\spec(\delta_{\triv_{\check Z}*})\in\IndCoh_\Nilp(\LS_{\check G})$ whose image under the geometric Langlands functor $\mathbb L_G^{-1}$ is quasi-lisse. Our goal is to show the image $\mathbb L_G^{-1}\circ\Poinc^\spec(\delta_{\triv_{\check Z}*})\in\Dmod(\Bun_G)$ is simply the constant sheaf $\omega_{\Bun_G}$. To do so, we compute the constant terms of both the spectral object $\Poinc^\spec(\delta_{\triv_{\check Z}*})$ and the automorphic object $\omega_{\Bun_G}$ and check they match under $\mathbb L_T$. This section is devoted to that computation.
\subsection{On the automorphic side} Consider the diagram of stacks
\[
\begin{tikzcd}
    &\Bun_{B^-}\arrow[swap]{dl}{\pp^-}\arrow{dr}{\qq^-}\\
    \Bun_G&&\Bun_T,
\end{tikzcd}
\]
where $\pp^-$ is quasi-compact and representable and $\qq^-$ is smooth \cite[Remark~0.2.9]{constant-term}. Let $\rr$ denote the morphism $\Bun_T\to\Bun_{B^-}$ induced from the inclusion $T\subset B^-$, which is a section of $\qq^-$.

We let $\CT^{-,\mathrm{nv}}_*:=\qq^-_*(\pp^-)^!$ and $\Eis_!^{-,\mathrm{nv}}:=(\pp^-)_!\circ(\qq^-)^*$ be the na\"ive constant term and Eisenstein series functors, and following \cite[\S8.1.6]{GLC3}, we introduce corrected versions,
\[\CT^-_*:=\CT^{-,\mathrm{nv}}_*[-\operatorname{dim.rel.}(\qq^-)],\ \ \Eis^-_!:=\Eis^{-,\mathrm{nv}}_![\operatorname{dim.rel.}(\qq^-)].\] 
Here, as usual, $\operatorname{dim.rel.}(\qq^-)$ is considered a locally constant function on $\Bun_T$.

\begin{prop}\label{prop:eisenstein-and-delta}
    There is an equivalence $\CT^-_*\simeq(\pp^-\circ\rr)^![\operatorname{dim.rel.}(\qq^-)]$ on the subcategory $\Dmod_{\fz(\fg)}(\Bun_G)\subset\Dmod(\Bun_G)$. 
\end{prop}
\begin{proof}
Let $\cF\in\Dmod_{\fz(\fg)}(\Bun_G)$ be arbitrary. As observed in \cite[\S4.1.6]{constant-term}, any object in $\Dmod(\Bun_{B^-})$ is automatically $\bG_m$-monodromic. Thus, by the contraction principle (see \cite[\S4.1]{constant-term} and \cite[Theorem~C.5.3]{compact-generation}) applied to $(\pp^-)^!\cF$,
\[
\CT^{-,\mathrm{nv}}_*\cF=\qq^-_*(\pp^-)^!\cF\simeq\rr^*(\pp^-)^!\cF.
\]
Moreover, $\pp^-$ is non-characteristic with respect to $\fz(\fg)\subset T^*\Bun_G$ in the sense of Definition~\ref{defn:non-char}. Indeed, the fiber of the co-differential $d(\pp^-)^\bullet$ at a $B^-$-bundle $\cP$ is simply $R\Gamma(X,\fg_\cP\otimes\Omega^1_X)\to R\Gamma(X,(\fg/\fn^-)_\cP\otimes\Omega^1_X)$ whose kernel on $H^0$ is $ H^0(X,\fn^-_\cP\otimes\Omega^1_X)$, which intersects $ H^0(X,\fz(\fg)\otimes\Omega^1_X)$ trivially. Thus, by Lemma~\ref{lem:SS-functorial}, the object $(\pp^-)^!\cF$ has singular support in $\fz(\fg)\subset T^*\Bun_{B^-}$, the locus of $T^*\Bun_{B^-}$ consisting of pairs $(\cP,\theta)$ of a $B^-$-bundle $\cP$ and $\theta\in \Gamma(X,(\fg/\fn^-)_\cP\otimes\Omega^1_X)$ such that $\theta$ lies in the image of $\Gamma(X,\fz(\fg)\otimes\Omega^1_X)$.

Furthermore, $\rr$ is non-characteristic with respect to $\fz(\fg)\subset T^*\Bun_{B^-}$. Indeed, the co-differential
   \[
   d\rr^\bullet\colon\Bun_T\times_{\Bun_{B^-}}T^*\Bun_{B^-}\to T^*\Bun_T
   \]
   on a $T$-bundle $\cP$ is $\Gamma(X,(\fg/\fn^-)_\cP\otimes\Omega^1_X)\to\Gamma(X,\ft_\cP\otimes\Omega^1_X)$, so $(d\rr^\bullet)^{-1}(0)$ classifies a $T$-bundle $\cP$ and a section $\phi\in\Gamma(X,\fn_\cP\otimes\Omega^1_X)$. This has trivial intersection with $\fz(\fg)\times_{\Bun_{B^-}}\Bun_T$. Thus again by Lemma~\ref{lem:SS-functorial},
   \[
   \rr^*(\pp^-)^!\cF\simeq\rr^!(\pp^-)^!\cF[-2\operatorname{dim.rel.}(\rr)]=(\pp^-\circ\rr)^!\cF[-2\operatorname{dim.rel.}(\rr)].
   \]   
   The corrected constant term is thus
   \[
   \CT^-_*\cF\simeq(\pp^-\circ\rr)^!\cF[-2\operatorname{dim.rel.}(\rr)-\operatorname{dim.rel.}(\qq^-)]
   \]
   Finally, since $\rr$ is a section of $\qq^-$ we see $\operatorname{dim.rel.}(\rr)=-\operatorname{dim.rel.}(\qq^-)$.
\end{proof}

\subsection{On the spectral side} 

\subsubsection{Spectral constant term}\label{subsec:spectral-constant} There is a diagram of stacks
\[
\begin{tikzcd}
    &\LS_{\check B^-}\arrow[swap]{dl}{\pp^{-,\spec}}\arrow{dr}{\qq^{-,\spec}}\\
    \LS_{\check G}&&\LS_{\check T},
\end{tikzcd}
\]
and as in \cite[\S12.1.1]{GLC3} we consider the spectral constant term functor
\[
\CT^{-,\spec}:=(\qq^{-,\spec})^{\IndCoh}_*\circ(\pp^{-,\spec})^!\colon\IndCoh_\Nilp(\LS_{\check G})\to\QCoh(\LS_{\check T}).
\]
We also consider the dual $\Eis^{-,\spec}=(\pp^{-,\spec})_*^{\IndCoh}\circ(\qq^{-,\spec})^!$.
A key step in the proof of the geometric Langlands equivalence is the compatibility between the automorphic constant term functor and the spectral constant term functor.
\begin{thm}[{\cite[Theorem~15.1.13]{GLC3}}]\label{thm:constant-term-compatibility}
    There is a commutative diagram
    \[\begin{tikzcd}
        \Dmod(\Bun_G)\arrow{r}{\mathbb L_G}\arrow[swap]{d}{\CT^-_{*,\rho(\omega_X)}[-d]}&\IndCoh_\Nilp(\LS_{\check G})\arrow{d}{\CT^{-,\spec}}\\
        \Dmod(\Bun_T)\arrow{r}{\mathbb L_T}&\QCoh(\LS_{\check T}),
    \end{tikzcd}
    \]
    where $d=\delta_{(N^-)_{\rho(\omega_X)}}=\dim(\Bun_{N^-,\rho(\omega_X)})$ as in \cite[\S9.6.5]{GLC2} and $\CT^-_{*,\rho(\omega_X)}=t_{\rho(\omega_X)}\circ\CT^-_*$ where $t_{\rho(\omega_X)}$ is the translation-by-$\rho(\omega_X)$ functor.
\end{thm}
Now, the following is the spectral counterpart of Proposition~\ref{prop:eisenstein-and-delta}. Before stating the isomorphism, note that $\QCoh(\LS_{\check T})$ carries a $\pi_1^{\mathrm{alg}}(T)$-grading, determined by the evaluation morphism $\operatorname{ev}_x\colon\LS_{\check T}\to\B\check T$ for a choice of base point $x\in X(e)$. Equivalently, under the (enhanced) Fourier-Mukai equivalence \eqref{eq:FM-enhanced} this is the grading by connected components of $\Bun_T$. In particular, for any element $\lambda\in\pi_1^{\mathrm{alg}}(\check T)$, the category $\QCoh(\LS_{\check T})$ inherits a $\Z$-grading.
\begin{prop}\label{prop:spectral-image}   
Let $\gamma$ be the morphism $\LS_{\check Z}\to\LS_{\check T}$. There is a natural isomorphism of functors $\QCoh(\LS_{\check Z})\to\QCoh(\LS_{\check T})$
\[\CT^{-,\spec}\circ\Poinc^\spec\simeq\shear\circ\gamma_*.\]
Here $\shear$ is the auto-equivalence of $\QCoh(\LS_{\check T})$ which shears with respect to the $\Z$-grading on $\QCoh(\LS_{\check T})$ induced by $2\check\rho$.
\end{prop}
\begin{proof}
    Let us pass to the dual, so we will compute $\Whit^{\spec}\circ\Eis^{-,\spec}$. Then
    \[
    \Whit^\spec\circ(\pp^{-,\spec})_*^{\IndCoh}\circ(\qq^{-,\spec})^!=\sing_{\LS_{\check Z}}\circ\Sigma_*^{\IndCoh}\circ\iota^!\circ(\pp^{-,\spec})_*^{\IndCoh}\circ(\qq^{-,\spec})^!.
    \]
Let $\mathfrak X:=\LS_{\check B^-}\times_{\LS_{\check G}}(\LS_{\check Z\check N}/\check\rho(\bG_m))$, which sits in the diagram 
\[
\begin{tikzcd}
    \mathfrak X\arrow{r}{\mathrm{pr}_2}\arrow[swap]{d}{\mathrm{pr}_1}&\LS_{\check Z\check N}/\check\rho(\bG_m)\arrow{r}{\Sigma}\arrow{d}{\iota}&(\LS_{\check Z}\times\Omega\bA^1)/\bG_m\\
    \LS_{\check B^-}\arrow{r}{\pp^{-,\spec}}\arrow[swap]{d}{\qq^{-,\spec}}&\LS_{\check G}\\
    \LS_{\check T}.
\end{tikzcd}
\]
By base change, \cite[\S5.2.5]{indcoh}, we see
\[
\Sigma_*^{\IndCoh}\circ\iota^!\circ(\pp^{-,\spec})_*\circ(\qq^{-,\spec})^!\simeq(\Sigma\circ\mathrm{pr}_2)^{\IndCoh}_*\circ(\qq^{-,\spec}\circ\mathrm{pr}_1)^!.
\]
The pullback $(\qq^{-,\spec}\circ\mathrm{pr}_1)^!$ will factor through $\QCoh(\mathfrak X)$ by \cite[Theorem~1.3.11]{singular-support}.

Recall that the stack $\LS_{\check B^-}\times_{\LS_{\check G}}\LS_{\check N}$
carries a Bruhat stratification into locally closed substacks $\LS_{\check N_w}$, where $\check N_w:=\check N\cap w\check B^-w^{-1}$ for any $w\in W$.\footnote{The stack $\check N\backslash\check G/\check B^-$ admits a stratification into the locally closed substacks $\pt/\check N_w$. This is the global analog. See for example \cite[\S15.3.2]{GLC3}.} Similarly, the stack $\mathfrak X$ carries a Bruhat stratification into locally closed substacks $\LS_{\check Z\check N_w}/\check\rho(\bG_m)$, where the open piece is $\LS_{\check Z}/\check\rho(\bG_m)$ corresponding to $w=1$. Let $j$ denote the open immersion of $\LS_{\check Z}/\check\rho(\bG_m)$ into $\mathfrak X$. Then the cone of
\begin{equation}\label{eq:open-strata}
(\qq^{-,\spec}\circ\mathrm{pr}_1)^!\to j_*^{\IndCoh}j^!(\qq^{-,\spec}\circ\mathrm{pr}_1)^!
\end{equation}
is filtered with subquotients which are quasi-coherent sheaves scheme-theoretically supported on $\LS_{\check Z\check N_w}/\check\rho(\bG_m)$ where $w\ne1$. Then the singular co-differential of the morphism
\[\sigma\colon\LS_{\check Z\check N_w}/\check\rho(\bG_m)\to\mathfrak X\xrightarrow{\Sigma\circ\pr_2}(\LS_{\check Z}\times\Omega\bA^1)/\bG_m\] at a $\check Z\check N_w$-local system $(\cP,\nabla)$ on $X\times S$ is (letting $\mathfrak z(\check\fg)=\Lie(\check Z)$)
\[
\Sing(\sigma)_{(\cP,\nabla)}\colon\Gamma(X^\dR\times S,\fz(\check\fg)^*)\times\cO_S\to\Gamma(X^\dR\times S,\fz(\check\fg)^*)\times\Gamma(X^\dR\times S,\operatorname{Lie}(\check N_w)_\cP^*)
\]
which is the identity on the first factor and is induced by the non-degenerate character $\check N_w\subset\check N\to\bG_a$ on the second factor. Whenever $w\ne1$ there exists a simple root $\alpha$ such that $w(\alpha)$ is negative, hence the composition $\check N_w\to\bG_a$ is nonzero. In particular $\Sing(\sigma)_{(\cP,\nabla)}$ is injective, so by the functoriality of singular support \cite[Theorem~1.3.11]{singular-support} the functor $\sigma_*^{\IndCoh}$ maps the subcategory $\QCoh(\LS_{\check Z\check N_w}/\check\rho(\bG_m))$ to the subcategory $\QCoh((\LS_{\check Z}\times\Omega\bA^1)/\bG_m)$. In particular, $\sing_{\LS_{\check Z}}\circ\sigma_*^{\IndCoh}=0$.

Thus, upon applying $\sing_{\LS_{\check Z}}\circ(\Sigma\circ\mathrm{pr}_2)^{\IndCoh}_*$ to \eqref{eq:open-strata} we obtain an isomorphism
\[
\sing_{\LS_{\check Z}}\circ(\Sigma\circ\mathrm{pr}_2)^{\IndCoh}_*\circ(\qq^{-,\spec}\circ\mathrm{pr}_1)^!\simeq\sing_{\LS_{\check Z}}\circ(\Sigma\circ\mathrm{pr}_2\circ j)^{\IndCoh}_*\circ(\qq^{-,\spec}\circ\mathrm{pr}_1\circ j)^!
\]
Here $\alpha:=\qq^{-,\spec}\circ\pr_1\circ j$ is the morphism $\mathrm{pr}_1\circ j\colon\LS_{\check Z}/\check\rho(\bG_m)$ is induced from the inclusion $\check Z\subset\check T$ and $\beta:=\Sigma\circ\mathrm{pr}_2\circ j$ is the morphism $\LS_{\check Z}/\check\rho(\bG_m)\to(\LS_{\check Z}\times\Omega\bA^1)/\bG_m$ induced from the obvious morphism $\pt\to\Omega\bA^1$. We have thus far shown
\[
\Whit^\spec\circ\Eis^{-,\spec}\simeq\sing_{\LS_{\check Z}}\circ\beta^{\IndCoh}_*\circ\alpha^!,
\]
so dually,
\[
\CT^{-,\spec}\circ\Poinc^\spec\simeq\alpha_*^{\IndCoh}\circ\beta^!\circ\exp_{\LS_{\check Z}}^\spec.
\]
By Lemma~\ref{lem:koszul-commutative} with $\bG_m$-equivariance, there is a commutative diagram
\[
\begin{tikzcd}
    \QCoh(\LS_{\check Z})\arrow{r}{\exp^\spec}\arrow{d}{=}&\IndCoh_{0\times\bA^1}\!\big((\LS_{\check Z}\times\Omega\bA^1)/\bG_m\big)\arrow{r}{\beta^!}\arrow{d}{\Kos}&\arrow{d}{=}\QCoh(\LS_{\check Z}/\check\rho(\bG_m))\\
    \QCoh(\LS_{\check Z})\arrow{r}{(j_{\bA^1})_*^\shear}&\QCoh((\LS_{\check Z}\times\bA^1)/\bG_m)^\shear\arrow{r}{(\mathrm{pr}_2)_*^\shear}&\QCoh(\LS_{\check Z}/\check\rho(\bG_m)).
\end{tikzcd}
\]
Next, as noted in \cite[Example~6.3.10]{relative-langlands} there is a commutative diagram
\[
\begin{tikzcd}
    \QCoh(\LS_{\check Z})\arrow{rr}{(\pr_2\circ j_{\bA^1})_*^\shear}\arrow{d}{=}&&\QCoh(\LS_{\check Z}/\check\rho(\bG_m))\arrow{d}{\shear}\arrow{r}{\alpha_*}&\QCoh(\LS_{\check T})\arrow{d}{\shear}\\
    \QCoh(\LS_{\check Z})\arrow{rr}{(\pr_2\circ j_{\bA^1})_*}&&\QCoh(\LS_{\check Z}/\check\rho(\bG_m))\arrow{r}{\alpha_*}&\QCoh(\LS_{\check T}),
\end{tikzcd}
\]
where $\shear$ is with respect to the $\Z$-grading on $\QCoh(\LS_{\check T})$ arising from $2\check\rho$.
\end{proof}

\section{Concluding the proof} \label{sec:conclusion}
We will combine the ingredients thus far to compute the image of the constant sheaf.

\subsection{The image of the constant sheaf}
\begin{lemma}\label{lem:conservative-and-t-exact}
    The functor $(\pp^-\circ\rr)^![2\delta_N]\colon\Dmod_0(\Bun_G)\to\Dmod_0(\Bun_T)$ is conservative and t-exact.
\end{lemma}
\begin{proof}
    By \cite[Lemma~4.1.5.1 and Remark~4.1.5.2]{quasi-temp} any $\cF\in\Dmod_0(\Bun_G)$ is regular holonomic, and $\pp^-\circ\rr$ is clearly non-characteristic with respect to $0$, so the functor $(\pp^-\circ\rr)^![2\delta_N]$ is t-exact (note $2\delta_N=\delta_G-\delta_T$). The functor is conservative since $\pp^-\circ\rr\colon\Bun_T\to\Bun_G$ is surjective on connected components. Indeed, the connected components of $\Bun_G$ are indexed by $\pi_{1,\mathrm{alg}}(G)$ and $\pi_{1,\mathrm{alg}}(T)\to\pi_{1,\mathrm{alg}}(G)$ is surjective. 
\end{proof}
We may finally compute the image of the constant sheaf.
\begin{thm}\label{prop:image-of-constant-sheaf}
    There is an isomorphism
    \[
    \mathbb L_G(\omega_{\Bun_G})\simeq\Poinc^\spec(\delta_{\triv_{\check Z}*}).
    \]
\end{thm}
\begin{proof}
To keep track of cohomological shifts, note that $\operatorname{dim.rel.}(\qq^-)=\delta_N+\langle2\check\rho,\lambda\rangle$ on a component $\Bun_T^\lambda$ of $\Bun_T$ by \cite[\S7.2.2]{GLC5}. Thus, if $\shear$ is the shearing auto-equivalence of $\QCoh(\LS_{\check T})$ as in Proposition~\ref{prop:spectral-image}, then
\[
\mathbb L_T^{-1}\circ\shear\simeq\mathbb L_T^{-1}[\operatorname{dim.rel.}(\qq^-)-\delta_N].
\]
Further, note that the translation functor $t_{-\rho(\omega_X)}$ interacts with shearing as
\begin{equation}\label{eq:coh-shifts}
t_{-\rho(\omega_X)}\circ\mathbb L_T^{-1}\circ\shear\simeq\mathbb L_T^{-1}[\operatorname{dim.rel.}(\qq^-)-\delta_N-\langle2\check\rho,(2g-2)\rho\rangle]=\mathbb L_T^{-1}[\operatorname{dim.rel.}(\qq^-)-\delta_{(N^-)_{\rho(\omega_X)}}].
\end{equation}
Now, by Corollary~\ref{cor:poincare-factors2} we know $\mathbb L_G^{-1}\circ\Poinc^\spec(\delta_{\triv_{\check Z}*})\in\Dmod_0(\Bun_G)$. Its $!$-pullback under $\pp^-\circ\rr$ is
\begin{align*}
    (\pp^-\circ\rr)^!\circ\mathbb L_G^{-1}\circ\Poinc&^\spec(\delta_{\triv_{\check Z}*})\simeq t_{-\rho(\omega_X)}\circ\CT^-_{*,\rho(\omega_X)}\circ \mathbb L_G^{-1}\circ\Poinc^\spec(\delta_{\triv_{\check Z}*})[-\operatorname{dim.rel.}(\qq^-)]\\
    &\simeq t_{-\rho(\omega_X)}\circ\mathbb L_T^{-1}\circ\CT^{-,\spec}\circ\Poinc^\spec(\delta_{\triv_{\check Z}*})[\delta_{(N^-)_{\rho(\omega_X)}}-\operatorname{dim.rel.}(\qq^-)]\\
    &\simeq t_{-\rho(\omega_X)}\circ\mathbb L_T^{-1}\circ\shear(\delta_{\triv_{\check T}*})[\delta_{(N^-)_{\rho(\omega_X)}}-\operatorname{dim.rel.}(\qq^-)]\\
    &\simeq t_{-\rho(\omega_X)}\circ\mathbb L_T^{-1}(\delta_{\triv_{\check T}*})\\
    &\simeq \omega_{\Bun_T}.
\end{align*}
Here, the first isomorphism is Proposition~\ref{prop:eisenstein-and-delta}, the second isomorphism is Theorem~\ref{thm:constant-term-compatibility}, the third isomorphism is Proposition~\ref{prop:spectral-image}, the fourth isomorphism is \eqref{eq:coh-shifts}, and the last isomorphism was discussed in Example~\ref{ex:constant-D-module-under-GLC-tori}.

In particular,
\[
(\pp^-\circ\rr)^!\circ\mathbb L_G^{-1}\circ\Poinc^\spec(\delta_{\triv_{\check Z}*})[-\delta_T]\simeq\omega_{\Bun_T}[-\delta_T]\in\Dmod(\Bun_T)^\heartsuit.
\]
    By Lemma~\ref{lem:conservative-and-t-exact}, the shifted pullback $(\pp^-\circ\rr)^![2\delta_N]$ is conservative and t-exact, so in fact \begin{equation}\label{eq:Poinc-in-dmod}\mathbb L_G^{-1}\circ\Poinc^\spec(\delta_{\triv_{\check Z}*})[-2\delta_N-\delta_T]\in\Dmod_0(\Bun_G)^\heartsuit.\end{equation} 
    Moreover, since $\Bun_T\to\Bun_G$ is surjective on the fundamental groups of each connected component, the object~\eqref{eq:Poinc-in-dmod} must simply be $\omega_{\Bun_G}[-\delta_G]$.
\end{proof}
\begin{remark}
    Theorem~\ref{prop:image-of-constant-sheaf} is consistent with the Hecke eigen-property of $\omega_{\Bun_G}$. Let us assume $\check Z$ is trivial in this remark, for simplicity. By \cite[Theorem~5]{derived-satake} for any representation $V$ of $\check G$ and $x\in X$, the sheaf $T_{V,x}(\omega_{\Bun_G})$ admits a filtration whose associated graded is $\bigoplus_{\lambda\in\Lambda}\omega_{\Bun_G}\otimes V_\lambda[-\langle2\check\rho,\lambda\rangle]$ where $\Lambda$ is the character lattice of $\check T\subset \check G$ and $V_\lambda$ is the $\lambda$-weight space of $V$. In terms of the spectral Poincar\'e sheaf, let $\mathrm{ev}_x\colon\LS_{\check G}(X)\to\mathbb B\check G$ be the evaluation at $x$. Then by the projection formula
    \[
        \mathrm{ev}_x^!V\otimes\Poinc^\spec(\delta_{\triv_{\check Z}*})\simeq\iota^{\IndCoh}_*\big(\iota^!\mathrm{ev}_x^!V\otimes\Sigma^!\exp^\spec\big).
    \]
    Now $\iota^!\mathrm{ev}^!_xV$ is the pullback along the evaluation morphism $\LS_{\check N}(X)/\check\rho(\bG_m)\to\mathbb B(\check N\rtimes_{\check\rho}\bG_m)$. Since $\check N$ is unipotent, the sheaf $\iota^!\mathrm{ev}^!_xV$ admits a filtration whose associated graded is $\bigoplus_{\lambda\in\Lambda}V_\lambda\otimes\cO(\langle\lambda,\check\rho\rangle)$ where $\cO(1)$ denotes the tautological line bundle on $\B\bG_m$. Since $\exp^\spec\otimes\cO(-1)\simeq\exp^\spec[2]$, our two formulas are compatible.
\end{remark}
\subsection{Conclusion of the proof}\label{subsec:main-proof}
\begin{lemma}\label{lem:spectral-poincare-is-linear}
    The functor $\Poinc^\spec$ is $\QCoh(\LS_{\check Z})$-linear with respect to the convolution module structures on $\QCoh(\LS_{\check Z})$ and $\IndCoh_\Nilp(\LS_{\check G})$.
\end{lemma}
\begin{proof}
    Recall that $\Poinc^\spec$ is defined to be the composition \eqref{eq:spectral-Poinc-defn}. We check that each functor is $\QCoh(\LS_{\check Z})$-linear. The functor
    $\exp^\spec_{\LS_{\check Z}}$ is $\QCoh(\LS_{\check Z})$-linear almost by definition. The functors $\iota_*^{\IndCoh}$ and $\Sigma^!$ are $\QCoh(\LS_{\check Z})$-linear since both morphisms $\iota$ and $\Sigma$ intertwine the action of the group stack $\LS_{\check Z}$.
\end{proof}

Now, we are finally ready to prove our main theorem, Theorem~\ref{thm:laf1}.
\begin{proof}[Proof of Theorem~\ref{thm:laf1}]
For any $\cF\in\Dmod(\Bun_{G_\abc})$ we see
    \begin{align*}
        \Poinc^\spec\circ\mathbb L_{G_\abc}(\cF)&\simeq\mathbb L_{G_\abc}(\cF)\star\Poinc^\spec(\delta_{\triv_{\check Z}*})\\
        &\simeq\mathbb L_{G_\abc}(\cF)\star\mathbb L_G(\omega_{\Bun_G})\\
        &\simeq\mathbb L_G(\pi_{\Bun_G}^!\cF\otimes^!\omega_{\Bun_G})\\
        &\simeq\mathbb L_G(\pi_{\Bun_G}^!\cF),
    \end{align*}
    where the first isomorphism is by Lemma~\ref{lem:spectral-poincare-is-linear}, the second isomorphism is by Theorem~\ref{prop:image-of-constant-sheaf}, and the third isomorphism is by Theorem~\ref{thm:GLC-intertwines-actions}.
\end{proof}

\bibliographystyle{amsalpha}
\bibliography{bibfile}

\end{document}